\documentclass[12pt]{article}

\usepackage{amssymb,amsmath,amscd,amsthm,  mathabx}

\usepackage{graphicx,epsfig, color,float}

\usepackage{graphicx}
\usepackage[active]{srcltx}

\newtheorem{theorem}{Theorem}[section]

\newtheorem{lemma}[theorem]{Lemma}

\newtheorem{remark}[theorem]{Remark}

\date{}

\begin{document}

\date{}
\title{Negative eigenvalues of non-local Schr\"{o}dinger operators with sign-changing potentials }
\author{S. Molchanov \footnote{Dept of Mathematics and Statistics, UNCC, Charlotte, NC 28223,  smolchan@uncc.edu}, B. Vainberg \footnote{Dept of Mathematics and Statistics, UNCC, Charlotte, NC 28223,  brvainbe@uncc.edu} }

\maketitle

\begin{abstract}
Simon's results on the negative spectrum of recurrent Schr\"{o}dinger operators ($d=1,2$) are extended to a wider class of potentials and to non-local operators. An example of $L^1-$potental is constructed for which the essensial spectrum  of two-dimensional Schr\"{o}dinger operator covers the whole axis. Some counterexamples are provided for transient operators ($d\geq3$) showing that the assumptions on the potential for the validity of the Cwikel-Lieb-Rozenblum estimate can't be improved significantly.

\end{abstract}
MSC2020: 35J10, 35P99, 35Q99, 47A10, 35R11
\section{Introduction}
Let
\begin{equation}\label{scho}
H=-\Delta-\sigma V(x), \quad x\in \mathbb{R}^d,
\end{equation}
be a Schr\"{o}dinger operator in $L^2(\mathbb R^d)$ with a real-valued potential $V$. If $V$ decays at infinity in an appropriate sense, then $Sp_{ess}=[0,\infty)$ and the negative spectrum is discrete. If $V(x)\geq 0$, then, under some conditions on  $V(x)$ (to be discussed later), negative eigenvalues exist for arbitrary $\sigma>0$ when the dimension $d$ is $1 $ or $2$, see \cite{simon, ckmv, puri}. The situation is absolutely different in higher dimensions. If $d\geq 3$  and
\begin{equation}\label{clr}
0<\int_{x:~V(x)>0}V(x)^{d/2}dx<\infty,
\end{equation}
then there is a critical value $\sigma_{cr}>0$ such that the operator $H$ does not have negative spectrum when $\sigma\leq\sigma_{cr}$, and negative eigenvalues exist (and their number $N(\sigma V)$ is finite) when $\sigma>\sigma_{cr}$ \cite{simon}.  The latter result is based on the Cwikel-Lieb-Rozenblum (CLR) estimate \cite{cw,li,ro,clr}. The same conditions ($d\geq3$ and (\ref{clr})) are needed for the validity of quasi-classical asymptotics of  $N(\sigma V)$ as $\sigma\to \infty$, \cite{rs}.

This fundamental difference between the low dimensional ($d=1,2$) and higher dimensional ($d\geq3$) Schr\"{o}dinger operators is related to the different asymptotic behavior of the Markov semigroups $T_t=\exp(t\Delta)$ and the corresponding Markov processes (Brownian motions) $b(t)$ as $t\to\infty$.  When $d=1,2$, the Brownian motion $b(t)$ is recurrent, and it is transient if $d\geq3$. The processes $b(t)$ is recurrent or transient if the integral
\begin{equation}\label{fff}
\int_0^\infty dt\int_{\mathbb R^d\times\mathbb R^d} p(t,x,y)f(x)f(y)dxdy=\int_0^\infty dt \int_{\mathbb R^d} e^{-tk^2}|\hat{f}(k)|^2dk=\int_{\mathbb R^d} \frac{|\hat{f}(k)|^2}{|k|^2}dk
\end{equation}
diverges or converges, respectively, when $\hat {f}(0)\neq 0$. Here $f\in C_0^\infty,~p(t,x,y)$ is the integral kernel of the semigroup $T_t=\exp(t\Delta)$, i.e., the fundamental solution of the heat equation.

Two types of results are obtained in this paper. The recurrent operators are studied in the first part of the paper. A couple of examples concerning transient operators are given in the second part of the paper. In the recurrent case, our main goal is to simplify the proofs and extend the classical results \cite{simon} of Simon on the existence of negative eigenvalues to a wider class of operators. Simon proved the existence of negative eigenvalues for standard Schr\"{o}dinger operators (\ref{scho}) with arbitrarily small $\sigma>0$ when $d=1,2$ and the following conditions hold:
\begin{equation}\label{pnon}
\int_{\mathbb R^d}V(x)dx\geq 0,
\end{equation}
\begin{equation}\label{ad1}
\int_{\mathbb R}(1+|x|^2)|V(x)|dx<\infty, \quad {\rm if} \quad d=1,
\end{equation}
and there are $\varepsilon,\delta>0$ such that
\begin{equation}\label{ad2}
\int_{\mathbb R^2}(1+|x|^\varepsilon)|V(x)|dx<\infty, \quad  \int_{\mathbb R^2}|V(x)|^{1+\delta}dx<\infty, \quad {\rm if} \quad d=2.
\end{equation}

We will assume that $H=A-\sigma V(x)$, where the unperturbed operator $A$ may be non-local and has the form of multiplication by a real-valued continuous function $a(k)$ after the application of the Fourier transform:
\begin{equation}\label{opl}
\widehat{A\psi}=a(k)\widehat{\psi}(k),
\end{equation}
where $a(k)$ has the following properties: $a(k)>0$ for $k\neq 0$, and $~a(k)\leq C|k|^d , \quad k\to 0.$ Additional restrictions on $V$ and $A$ will be imposed when the integral in (\ref{pnon}) is zero.


Our conditions on $a(k)$ allow $A$ to be a generator of a symmetric homogeneous in space and time Levy process. Function $a$ in this case has the form
\[
a(k)=\int_{\mathbb \mathbb{R}^d}(1-\cos<k,z>)l(z)dz,
\]
where function $l$ (the density of the Levy measure associated with the semigroup $\exp(tL)~$) has the properties
\[
l(z)=l(-z),~\quad l(z)\geq0,~~\quad\int_{|z|>1}l(z)dz<\infty,~~\quad\int_{|z|<1}|z|l(z)dz<\infty.
\]
A typical example of such a generator $A$ is provided by a fractional power of the Laplacian $A=(-\Delta)^{\alpha/2},~0<\alpha\leq2$. We will consider this example only when $d=2,\alpha=2,$ or $d=1,~1\leq\alpha\leq2,$ to guarantee the recurrence of the process. The process is transient if $d=1,~0<\alpha<1,$ or $d=2,~0<\alpha<2$.

In general, the operator (\ref{opl}) is not Markovian, and the kernel $ p(t,x,y)$ of the corresponding semigroup $T(t)=e^{-tA}$ is not positive. However, conditions imposed on $a(k)$ imply that the integral
\[
\int_0^\infty dt\int_{\mathbb R^d\times\mathbb R^d} p(t,x,y)f(x)f(y)dxdy=\int_0^\infty dt \int_{\mathbb R^d} e^{-ta(k)}|\hat{f}(k)|^2dk=\int_{\mathbb R^d} \frac{|\hat{f}(k)|^2}{a(k)}dk
\]
diverges if $f\in C_0^\infty$, $\hat{f}(0)\neq0$. By analogy to (\ref{fff}), we will call this operators recurrent.

Our first result provides  the existence of a negative spectrum for the general recurrent operators $H=A-\sigma V(x)$, where $A$ is defined by (\ref{opl}), $V\in L^1$, and the integral of $V$ is strictly positive (conditions (\ref{ad1}), (\ref{ad2}) are not imposed). A short proof (several lines) of this result is based on a study of the quadratic form of $H$ in the Fourier space. Note that the condition $V\in L^1$ does not always guaratee the discreteness of the spectrum. Some sufficient conditions for the negative spectrum to be discrete are given in Theorem \ref{tt2}.  Condition (\ref{pnon}) will be omitted there. Thus these theorems together provide conditions for existence of the negative eigenvalues when $\int V(x)dx>0$.

When the integral in (\ref{pnon}) is zero, our proof of the existence of negative eigenvalues is more technical compared to the case when the integral is positive. We provide the proof for operator (\ref{scho}) when $d=2$ with slightly weaker conditions than in (\ref{ad2}) and for the perturbations of a fractional power of the Laplacian when $d=1$. The next result of the first part of the paper (Section 3) is the proof of the existence of negative eigenvalues for all $\sigma>0$ for $1$-dimensional operator (\ref{scho}) on the half-axis $x>0$ with the Neumann boundary condition at $x=0$ under Bargman's condition on $V$:
 \begin{equation}\label{be}
\int_{o}^\infty x|V(x)|dx<\infty.
\end{equation}
In particular, this result provides the existence of the negative eigenvalues for all $\sigma>0$ for operators on the whole line $\mathbb R$ when the potential is even. The latter case includes the class of positive-definite potentials that appear in many important practical applications,
where $V$ has the form of a convolution: $V(x)=b(x)\ast b(-x)$. Hence, $V$ is even and
\[
\int_{\mathbb R^d}V(x)dx=(\int_{\mathbb R^d}b(x)dx)^2\geq 0.
\]
Note that the Bargman condition (\ref{be}) implies (see \cite{simon}) that the negative eigenvalue is unique when $\sigma$ is small enough.

Section 4 contains an example of $L^1-$potental\ for which the esensial spectrum of two dimensional  Schr\"{o}dinger operator covers the whole axis.

Two examples are provided in the second part of the paper. They concern the transient classical Schr\"{o}dinger operator (\ref{scho}) with $d\geq3$ for which condition (\ref{clr}) is slightly violated. The first example shows that it is possible to have infinitely many negative eigenvalues when $V(x)$ decays rather fast at infinity, but the integral in  (\ref{clr}) diverges. The potential $V$ in this example has order $O(|x|^{-2}), ~|x|\to\infty,$ on a sequence of non-intersecting balls of increasing radii, and $V$ vanishes outside of the balls. Note that a logarithmically stronger estimate $|V(x)|<C(1+|x|\ln(2+|x|)^{-2}$ implies (\ref{clr}) and the boundedness of $N(\sigma V)$. The second example shows that $H$ may have no negative eigenvalues when a non-negative $V$ decays slowly at infinity.

\section{Perturbations of recurrent operators}

\begin{theorem}
Let $H=A-\sigma V(x)$ be an operator in $L^2( \mathbb{R}^d), d\geq 1,$ where operator $A$ is defined by (\ref{opl}) and function $a$ has the properties
\begin{equation}\label{akk}
a(k)>0  \quad ~{\it when} \quad k\neq 0, \quad ~a(k)\leq C|k|^d , \quad k\to 0.
\end{equation}
If  $V\in L^1(\mathbb{R}^d)$  and the integral of  $V$ is strictly positive:
$$
\int_{\mathbb{R}^d}V(x)dx> 0,
$$
then the negative spectrum of the operator $H=A-\sigma V$ with an arbitrarily small $\sigma>0$ is not empty.
\end{theorem}
\begin{proof}
For each $\psi\in L^2(\mathbb{R}^d)$, we have 
\begin{equation}\label{qf}
<H\psi,\psi>=\int_{\mathbb{R}^d}a(k)|\widehat{\psi}(k)|^2dk-
\sigma\int_{\mathbb{R}^d}\int_{\mathbb{R}^d}\widehat{V}(k-\xi)\widehat{\psi}(k)\overline{\widehat{\psi}(\xi)}dkd\xi:=I_1-\sigma I_2,
\end{equation}
where $\widehat{\psi}, \widehat{V}$ are the Fourier transforms of functions $\psi, V$, respectively.
One needs only to show that, for each $\sigma>0$, there is $\widehat{\psi}\in L^2(\mathbb{R}^d)$ for which the quadratic form above is positive.

Let us fix $\varepsilon_1>0$ such that $\widehat{V}(k-\xi)>v_0>0$ when $|k|,|\xi|<\varepsilon_1$ and the second estimate in (\ref{akk}) holds when $|k|\leq \varepsilon_1$. Since $\widehat{V}(0)>0$ and $\widehat{V}(k)$ is continuous, such $\varepsilon_1>0$ and $v_0$ exist. We define $\psi(x)=\psi_\varepsilon(x)$ by using its Fourier image:
\[
\widehat{\psi}_\varepsilon(k)=\left\{
                            \begin{array}{c}
                              |k|^{-d},~~ 0<\varepsilon\leq |k|\leq \varepsilon_1, \\
                              0,~~ |k|\notin(\varepsilon,\varepsilon_1),  \\
                            \end{array}  \right.
\]
where $\varepsilon\in(0,\varepsilon_1)$ will be chosen below. Then
\[
|I_1|\leq C\int_{\varepsilon<|k|<\varepsilon_1}|k|^{-d}dk<C_1|\ln\varepsilon|,  \quad \varepsilon\to0,
\]
\[
I_2>v_0\int_{\varepsilon<|k|<\varepsilon_1}\widehat{\psi}_\varepsilon(k)dk\int_{\varepsilon<|\xi|<\varepsilon_1}\widehat{\psi}_\varepsilon(\xi)d\xi=C_2v_0|\ln\varepsilon|^2>0,  \quad \varepsilon\to0.
\]
Hence the form (\ref{qf}) is negative if we choose sufficiently small $\varepsilon$ (such that $\sigma|\ln\varepsilon|\gg 1$).

\end{proof}

Some sufficient conditions for the negative spectrum of recurrent operators to be discrete are given below.
\begin{theorem}\label{tt2}
Let a) $H=A-\sigma V(x)$ be an operator in $L^2( \mathbb{R}^d), d\geq 1,$ where $A$ is defined by (\ref{opl}), function $a$ has the properties (\ref{akk}), $a(k)\to\infty$ as $|k|\to\infty$, and $|V(x)|$ is bounded and vanishes at infinity, or

b) $H=-\Delta-\sigma V(x)$ be an operator in $L^2( \mathbb{R}^2),$ where $V(x)$ satisfies the following condition:
\[
\int_{ \mathbb{R}^2} V_+(x)\ln(2+|x|)dx+\int_{ V>1}V(x)\ln V(x)dx<\infty, \quad  \quad  V_+(x)=\max (V(x), 0).
\]
Then the negative spectrum of $H$ is discrete.
\end{theorem}
\begin{remark}  Essential spectrum of two dimensional Schr\"{o}dinger operator with a potential from $L^1( \mathbb{R}^2)$ may cover the whole $\lambda$-axis, see Section 4.
\end{remark}
\begin{proof} When the last condition holds, the  statement is proved in \cite{barg} (with a slightly weaker assumption on $V$). Hence, we will consider here only the first case. 

Let the first condition hold and $VR_\lambda^0=V(x)(A-\lambda)^{-1}$. Operator $(H-\lambda)^{-1}, ~\lambda\notin[0,\infty),$ is bounded in $L^2( \mathbb{R}^d)$ if $I-\sigma VR_\lambda^0$ is invertible in $L^2( \mathbb{R}^d)$. Since
\[
\|VR_\lambda^0\|\leq\sup_x |V(x))|\sup_k\frac{1}{|a(k)-\lambda|},
\]
operator $VR_\lambda^0, ~\lambda\notin[0,\infty),$ is a limit as $N\to\infty$  of compact operators $\kappa(x/N)V(x)K_N(A-\lambda)^{-1}$ where $\kappa(x)\in C^\infty_0, ~\kappa(x)=1$ in a neighborhood of the origin, and operator $K_N$ after the Fourier transform is given by multiplication by the function $\kappa(k/N)$. Hence the operator $VR_\lambda^0, ~\lambda\notin[0,\infty),$ is compact. One can also easily check that this operator is analytic in $\lambda$, and $I-\sigma VR_\lambda^0$ is invertible for non-real $\lambda$ (and for $\lambda
\ll-1$). Hence, the statement follows immediately from the analytic Fredholm theorem. 

\end{proof}

Denote by $L$ the space of functions on $\mathbb R^2$ that are integrable with the weight $\ln(2+|x|)$, and let
\[
\|f\|_L=\int_{\mathbb R^2}\ln(2+|x|)|f(x)|dx.
\]
\begin{theorem}\label{t2}
Let $d=2$, and suppose that the following assumptions hold for some $\delta>0$: 
\begin{equation*}
\int_{\mathbb R^2}\ln^2(2+|x|)|V(x)|dx<\infty;~\quad \int_{\mathbb R^2}|V(x)|^{1+\delta}dx<\infty; \quad
\int_{\mathbb{R}^2}V(x)dx\geq 0.
\end{equation*}
Then the negative spectrum of  the operator $H=-\Delta-\sigma V$ is discrete, and $H$ has a negative eigenvalue $\lambda=\lambda(\sigma)$ for arbitrary $\sigma>0$.
\end{theorem}
\begin{proof} Theorem \ref{tt2} implies the discreteness of the negative spectrum. Thus it
is enough to justify the second statement regarding the existence of a negative
eigenvalue and prove it only when the integral of $V$ is zero. Since the operator $-\sigma^{-1}\Delta- V,~ \sigma>0,$ decreases monotonically when $\sigma$ increases, it is enough to show the existence of a negative eigenvalue of $H$ for arbitrarily small $\sigma>0$. This will imply its existence for all $\sigma>0$.

We will look for eigenfunctions of $H$ with $\lambda<0$ in the form $\psi=(-\Delta-\lambda)^{-1}f, ~f\in L$. Using the Fourier transform, one can check that  $\psi\in L^2(\mathbb R^2)$. This function also satisfies the equation $H\psi-\lambda\psi=0$ if $f$ is a solution of the equation
\begin{equation}\label{inteq2}
f(x)-\sigma V(x)\int_{\mathbb{R}^2}G(\lambda, x-y)f(y)dy=0, \quad f\in L, \quad \lambda<0,
\end{equation}
where $G$ is the integral kernel of the resolvent $(-\Delta-\lambda)^{-1}$, i.e., $G=\frac{1}{2\pi}K_0(\sqrt{-\lambda}|x-y|)$. Here $K_0$ is the modified Bessel function of the second kind, see \cite{leb}.
 Thus the theorem will be proved if we show that, for arbitrarily small $\sigma>0$, there exists $\lambda=\lambda(\sigma)<0$ for which equation (\ref{inteq2}) has a non-trivial solution $f\in L$.

 We fix a smooth function $h$ with a compact support such that $\int_{\mathbb{R}^2} hdx=1$, and write $f$ in the form
 \[
f=C_fh+f^\bot, \quad
{\rm where}~~~ C_f=\int_{\mathbb{R}^2}f(x)dx, \quad \int_{\mathbb{R}^2}f^\bot(x)dx=0.
 \]

We single out the singularity of $G$ as $\lambda\to-0$: $G= \frac{-1}{4\pi}\ln|\lambda|+g(\lambda, |x-y|)$, where $g+\frac{1}{2\pi}\ln|x-y|$ is continuous when $\lambda\leq0$ and there are constants $C,c$ such that
\begin{equation}\label{gl2}
 |g|\leq C(|\ln|x-y||+1), \quad  g(0,x-y)=\frac{-1}{2\pi}\ln|x-y|+c, \quad \lambda\leq0.
\end{equation}
We put $G= \frac{-1}{4\pi}\ln|\lambda|+g(\lambda, |x-y|)$ into (\ref{inteq2}) and rewrite it as a system for $C_f$ and $f^\bot$:
\begin{equation}\label{s122}
C_f-\sigma C_fB_\lambda h-\sigma B_\lambda f^\bot=0,
\end{equation}
\begin{equation}\label{s1232}
f^\bot(x)+C_f\frac{\sigma V(x)}{4\pi}\ln|\lambda|-\sigma C_f[A_\lambda h-(B_\lambda h)h]
-\sigma [A_\lambda f^\bot-(B_\lambda f^\bot)h]=0,
\end{equation}
where
\begin{align}\label{A02}
A_\lambda:L\to L, \quad A_\lambda f=V(x)\int_{\mathbb{R}^2}g(\lambda, x-y)f(y)dy, \quad \lambda\leq0, \\ B_\lambda:L\to\mathbb R, \quad B_\lambda f=\int_{\mathbb{R}^2}\int_{\mathbb{R}^2}V(x)g(\lambda, x-y)f(y)dydx,  \quad \lambda\leq0\label{B02}.
\end{align}

Operators $A_\lambda,B_\lambda$ are bounded uniformly in $\lambda\in[-1,0]$ and strongly continuous when $\lambda\in[-1,0]$.  Indeed, the first estimate in (\ref{gl2}) implies that
\begin{align}\nonumber
\|A_\lambda f\|_L  \leq C\int_{\mathbb{R}^2}\int_{\mathbb{R}^2}\ln(2+|x|)|V(x)|(|\ln| x-y||+1)f(y)|dydx \\ \leq C\int_{\mathbb{R}^2}(I_1(y)+I_2(y))|f(y)|dy,\quad I_i=\int_{D_i}\ln(2+|x|) |V(x)|(|\ln| x-y||+1)dx ,\label{bb2}
\end{align}
where $D_1=\{x:~|x-y|<1\}, ~D_2=\{x:~|x-y|>1\}$. We have
\begin{align*}
I_1\leq C\ln(2+|y|)\left(\int_{|x-y|<1}| V(x)|^{1+\delta}dx\right)^{\frac{1}{1+\delta}}\left(\int_{|x-y|<1}(|\ln| x-y||+1)^{\frac{1+\delta}{\delta}}dx\right)^{\frac{\delta}{1+\delta}}  \\ \leq C_1\ln(2+|y|),
\end{align*}
\[
I_2\leq C\int_{|x-y|>1}| V(x)|\ln^2(2+|x|)\ln(2+|y|)dx\leq C_2\ln(2+|y|).
\]
The estimates on $I_i$ together with (\ref{bb2}) prove the uniform boundedness of $A_\lambda$. The boundedness of $B_\lambda$ can be proved similarly. The same arguments complemented by the continuity of $g$ and the dominated convergence theorem can be used to prove the strong continuity of both operators.

The boundedness of $A_\lambda, B_\lambda$ allows us to rewrite (\ref{s122}), (\ref{s1232}) for small $\sigma\geq0$ and $\lambda\in [-1,0]$ in a slightly shorter form:
\begin{equation}\label{s12a2}
C_f[1+O(\sigma)] -\sigma B_\lambda f^\bot=0,
\end{equation}
\begin{equation}\label{s123a2}
f^\bot(x)+\mathbf{O}(\sigma)f^\bot+C_f[\frac{\sigma V(x)}{4\pi}\ln|\lambda|+\mathbf{O}_1(\sigma)]=0,
\end{equation}
where $\mathbf{O}(\sigma)$ is an operator in $L$ of order $\sigma$ and $\mathbf{O}_i(\sigma)$ are elements of $L$ of order $\sigma$. We solve equation (\ref{s123a2}) for $f^\bot$ and obtain:
\[
f^\bot=-C_f[I+\mathbf{O}(\sigma)]^{-1}(\frac{\sigma V(x)}{4\pi}\ln\lambda+\mathbf{O}_1(\sigma))=-C_f[\frac{\sigma \ln|\lambda|}{4\pi}(V(x)+\mathbf{O}_2(\sigma))+\mathbf{O}_3(\sigma)].
\]
We put the latter expression into (\ref{s12a2}) and obtain
\begin{equation}\label{cf2}
C_fp=0, \quad {\rm where} \quad p=p(\sigma,\lambda)=1+o(1)+\frac{\sigma^2 m}{4\pi}\ln|\lambda|(1+o(1)),
\end{equation}
with the remainder terms vanishing when $\sigma\to +0,~ \lambda\to-0$, and
\begin{equation}\label{dm}
m=B_0V(x)=\int_{\mathbb{R}^2}\int_{\mathbb{R}^2}V(x)(\frac{\ln|x-y|}{-2\pi}+c)V(y)dydx=\int_{\mathbb{R}^2}\int_{\mathbb{R}^2}V(x)
\frac{\ln|x-y|}{-2\pi}V(y)dydx.
\end{equation}
Since equation (\ref{inteq2}) is reduced to (\ref{cf2}), the values of $\lambda$ for which $p=0$ are the eigenvalues of  operator $H$.

From (\ref{dm}) and the estimate $\|A_0V\|_L\leq\infty$ it follows that $|m|<\infty$. Since the Fourier transform of the function $\frac{-1}{2\pi}\ln|x|$ is $1/|k|^2$,  the Plancherel identity implies that
\begin{equation}\label{mP2}
m=\int_{\mathbb R^2}\frac{|\widehat{V}(k)|^2}{|k|^2}dk>0.
\end{equation}

We fix $a\in(0,1)$ so small that the remainder terms in (\ref{cf2}) satisfy $|o(1)|<1/2$ when $-\lambda,\sigma\in[0,a] $. Then we fix $\sigma=\sigma'\in(0,a)$ as small as we please for which $d(\sigma',a)>0$. Since $\lim_{\lambda\to -0}d(\sigma', \lambda)=-\infty$, there exists $b\in(0,a)$ such that $d(\sigma', b)<0$. Since the strong continuity of operators $A_\lambda,B_\lambda$ implies the continuity of the remainder terms in (\ref{cf2}), there exists a point $\lambda =\lambda(\sigma')\in(-a,-b)$ for which $p(\sigma', \lambda)=0$. This point is an eigenvalue of $H$.

\end{proof}

Consider now a one-dimensional version of Theorem \ref{t2} in the case when the Laplacian is replaced by a more general operator. A very natural class of these operators is presented by fractional powers $A=(-\Delta)^{\alpha/2}$ of the Laplacian. We assume that $\alpha\leq2$, so that $L$ is a generator of a symmetric stable process with parameter $\alpha$, and that $\alpha> 1$, so that the process is recurrent. In fact, the process remains recurrent when $\alpha=1$, but the statement and some estimates in the proof require changes, and we decided to exclude this case from consideration.

\begin{theorem}\label{t3}
Let $d=1, ~H=(-\Delta)^{\alpha/2}-\sigma V(x), ~1<\alpha\leq2$, and suppose that the following assumptions hold:
\begin{equation*}
\int_\mathbb R(1+|x|^{2\alpha-2})|V(x)|dx<\infty, \quad
\int_{\mathbb{R}^d}V(x)dx\geq 0.
\end{equation*}
Then the negative spectrum of  the operator $H=-\Delta-\sigma V$ is discrete, and $H$ has a negative eigenvalue $\lambda=\lambda(\sigma)$ for each $\sigma>0$.
\end{theorem}
\begin{proof} The proof copies the one of Theorem \ref{t2} with some steps being simpler. Using Theorem 3.1 from \cite{barg}, one can prove that the negative spectrum of $H$ is discrete. Thus we can assume that the integral of $V$ is zero and justify the existence of a negative eigenvalue only for arbitrary small $\sigma\geq0$, see the first paragraph of the proof of Theorem \ref{t2}.

Space $L$ now consists of functions integrable with the weight $1+|x|^{\alpha-1}$ with
\[
\|f\|_{L}=\int_\mathbb R(1+|x|^{\alpha-1})|f(x)|dx.
\]
We reduce again the eigenvalue problem for operator $H$ to the one-dimensional version of equation (\ref{inteq2}) with $G(\lambda, x-y)$ being the integral kernel of the operator $((-\Delta)^{\alpha/2}-\lambda)^{-1}$.

We look at the singularity of $G$ as $\lambda\to-0$. Denote $\nu=|x||\lambda|^{1/\alpha}$. We have

\[
G(\lambda,x)=\frac{1}{2\pi}\int_{-\infty}^\infty\frac{e^{ikx}dk}{|\lambda|+|k|^\alpha}=\frac{1}{\pi}\int_{0}^\infty\frac{\cos(kx)dk}{|\lambda|+|k|^\alpha}=
\frac{|\lambda|^{\frac{1-\alpha}{\alpha}}}{\pi}\int_{0}^\infty\frac{\cos(s\nu)ds}{1+s^\alpha}=
\]
\begin{equation}\label{sing}
=\frac{|\lambda|^{\frac{1-\alpha}{\alpha}}}{\pi}\int_{0}^\infty\frac{ds}{1+s^\alpha}-
\frac{|\lambda|^{\frac{1-\alpha}{\alpha}}}{\pi}\int_{0}^\infty\frac{2\sin^2(s\nu/2)ds}{1+s^\alpha}:=c_1(\alpha)|\lambda|^{\frac{1-\alpha}{\alpha}}+g(\lambda,x), \quad c_1(\alpha)>0.
\end{equation}
We replace $1+s^\alpha$ in the last integrand by $s^\alpha$ and make the substitution $ s\nu\to s$. This leads to the estimate
\begin{equation}\label{ge}
|g(\lambda,x)|\leq C|\lambda|^{\frac{1-\alpha}{\alpha}}\nu^{\alpha-1}=C|x|^{\alpha-1}, \quad \lambda<0.
\end{equation}

Function $g$ is continuous for $\lambda>0$ and can be extended by continuity when $\lambda\to -0$ with
\begin{equation}\label{g00}
g(0,x)=c_2(\alpha)|x|^{\alpha-1}, \quad c_2(\alpha)>0.
\end{equation}
The latter statement follows from the relation
\[
g(\lambda,x)=\frac{|\lambda|^{\frac{1-\alpha}{\alpha}}}{\pi}\int_{0}^\infty\frac{2\sin^2(s\nu/2)ds}{s^\alpha}+
\frac{|\lambda|^{\frac{1-\alpha}{\alpha}}}{\pi}\int_{0}^\infty\frac{2\sin^2(s\nu/2)ds}{s^\alpha(1+s^\alpha)}
\]
\[
=\frac{|x|^{\alpha-1}}{\pi}\int_{0}^\infty\frac{2\sin^2(s/2)ds}{s^\alpha}+
\frac{|x|^{\alpha-1}}{\pi}\int_{0}^\infty\frac{2\sin^2(s/2)ds}{s^\alpha(1+(s/\nu)^\alpha)},
\]
where the first term is $c_2(\alpha)|x|^{\alpha-1}$ and the second one vanishes when $x$ is fixed and $\lambda\to-0$.

The relations (\ref{ge}), (\ref{g00}) are the analogue of (\ref{gl2}). Using (\ref{sing})-(\ref{g00}) and the  representation $f=C_fh+f^\bot, ~C_f=\int_{\mathbb{R}}f(x)dx$, we rewrite the one-dimensional version of equation (\ref{inteq2}) as the system (\ref{s122})-(\ref{B02}) with the singularity $\frac {\ln|\lambda|}{-4\pi}$ of $G$ in (\ref{s1232}) replaced by $c_1(\alpha)|\lambda|^{\frac{1-\alpha}{\alpha}}$ and the integration in (\ref{A02}), (\ref{B02}) over $\mathbb R$ instead of $\mathbb R^2$. The remaining steps of the proof are the same with the following analogue of formula (\ref{cf2}) for function $p$ defining the eigenvalues of $H$:
\[
p=p(\sigma,\lambda)=1+o(1)-c_1(\alpha)|\lambda|^{\frac{1-\alpha}{\alpha}}\sigma^2 m(1+o(1)), \quad m=\int_{\mathbb R^2}\frac{|\widehat{V}(k)|^2}{|k|^\alpha}dk>0.
\]

\end{proof}

\section {ODE on the half-line.}

Consider operator $H$ in $L^2(\mathbb R_+)$ defined by the following relations
\begin{equation}\label{ode1}
H\psi=-\psi''(x)-\sigma V(x)\psi(x), \quad x>0, \quad \psi'(0)=0, \quad {\rm where}    \quad \int_0^\infty x|V(x)|dx<\infty.
\end{equation}
\begin{theorem}\label{tl}
If $\int_0^\infty V(x)dx\geq0$, then operator $H$ defined in (\ref{ode1}) has a negative eigenvalue $\lambda=\lambda(\sigma)$ for each $\sigma>0$.
\end{theorem}
We will need a couple of lemmas. The following statement can be found in \cite{st} under a more restrictive assumption on $V$. It is referred to as the Schpat theorem there.
\begin{lemma}
Equation
\begin{equation}\label{ode}
\psi''(x)+\sigma V(x)\psi(x)=0, \quad \int_0^\infty x|V(x)|dx<\infty,
\end{equation}
 has solutions $\psi_1,\psi_2$ with the following asymptotic behavior at infinity:
\begin{equation}\label{as}
\psi_1(x)=1+o(1), \quad \psi_1'(x)=o(1), \quad  \psi_2(x)=x(1+o(1)),\quad \psi_2'(x)=1+o(1), \quad x\to\infty.
\end{equation}
\end{lemma}
\begin{proof}
Denote
\[
\alpha=|\sigma|\int_0^\infty x|V(x)|dx.
\]
Since the statement of the lemma does not depend on the choice of the origin on the $x$-axis, without loss of the generality, we can assume that $\alpha<1$.

We look for $\psi_1$ in the form $\psi_1=1+z(x)$. Then $z''+\sigma V(x)z=-\sigma V(x)$, and this equation can be reduced to the integral equation
\begin{equation}\label{eq1}
z(x)+\sigma Pz=-\sigma\int_x^\infty (\xi-x)V(\xi) d\xi, \quad x>0, \quad {\rm where}    \quad Pz=\int_x^\infty (\xi-x)V(\xi)z(\xi)d\xi.
\end{equation}
If $z$ is a solution of (\ref{eq1}), then $\psi_1$ satisfies (\ref{ode}). We consider $P$ as an operator in the space $C$ of continuous bounded functions on the semi-axis $R_+$. Then $\|\sigma P\|\leq \alpha<1$, and therefore equation (\ref{eq1}) is uniquely solvable in $C$. Since $Pz$ vanishes at infinity for each $z\in C$, from (\ref{eq1}) it follows that $z$ vanishes at infinity. Differentiation of (\ref{eq1}) in $x$ implies  that $z'$ vanishes at infinity. Thus $\psi_1$ has the desired asymptotic behavior.

Using the solution $\psi=\psi_1$ of (\ref{ode}), one can obtain the general solution of (\ref{ode}) by looking for it in the form $\psi=\psi_1z$. In particular, the second solution can be chosen as
\begin{equation}\label{x0}
\psi_2(x)=\psi_1(x)\int_{x_0}^x \psi_1^{-2}(\xi)d\xi, \quad x\geq x_0,
\end{equation}
where $x_0\geq 0$ is an arbitrary point such that $\psi_1(x)\neq 0$ for $ x\geq x_0$. Such a point exists since $\psi_1(x)\to 1$ as $x\to \infty$. Then (\ref{as}) holds.

\end{proof}

\begin{lemma}
If $\int_0^\infty V(x)dx\geq 0$ and $\sigma>0$ is small enough, then the solution $\psi=\psi_0$ of (\ref{ode}) with the initial data $\psi_0(0)=1,~\psi_0'(0)=0$ vanishes at some point $x_0>0$.
\end{lemma}
\begin{proof}
We need initial data for $\psi_1,\psi_2$ when $\sigma\to 0$ and $V$ is fixed. Since the operator $(I+\sigma P) ^{-1}$ is analytic in $\sigma$ at $\sigma=0$, function $\psi_1$ as an element of $C$ is analytic at $\sigma=0$. Hence
\[
\psi_1=1+z=1-\sigma\int_x^\infty (\xi-x)V(\xi) d\xi+\sigma^2P\int_x^\infty (\xi-x)V(\xi) d\xi+O(\sigma^3).
\]
Equation (\ref{eq1}) implies that the Taylor series of $\psi_1$ in $\sigma$ can be differentiated in $x$. Thus, after differentiation of (\ref{eq1}) in $x$ followed by integration by parts, we get
\[
\psi_1'=\sigma\int_x^\infty V(\xi) d\xi-\sigma^2\int_x^\infty V(\eta)\int_\eta^\infty (\xi-\eta)V(\xi) d\xi d\eta +O(\sigma^3)
\]
\[
=\sigma\int_x^\infty V(\xi) d\xi-\sigma^2 \int_x^\infty V(\xi)d\xi\int_x^\infty (\xi-x)V(\xi) d\xi+\sigma^2\int_x^\infty \left(\int_\eta^\infty V(\xi)d\xi\right)^2d\eta+O(\sigma^3),
\]
and therefore
\[
\psi_1(0)=1-\sigma a+O(\sigma^2),\quad \psi_1'(0)=\sigma b+\sigma^2(-ab+c^2)+O(\sigma^3),
\]
where
\[
a= \int_0^\infty\xi V(\xi)d\xi, \quad b=\int_0^\infty V(\xi)d\xi, \quad c^2=\int_0^\infty \left(\int_\eta^\infty V(\xi)d\xi\right)^2d\eta.
\]
Since function $\psi_1$ as an element of $C$ is analytic in $\sigma$ at $\sigma=0$, we have $\psi_1=1+O(\sigma)\neq 0, ~x\geq 0, ~\sigma\ll 1$. Thus one can choose $x_0=0$ in (\ref{x0}), and therefore 
\[
 \quad \psi_2(0)=0, \quad \psi_2'(0)=1+\sigma a+O(\sigma^2).
\]

We represent $\psi_0$ as $\psi_0=A\psi_1+B\psi_2$. Initial data for these three solutions allow us to find $A$ and $B$:
\[
A=\frac{1}{1-\sigma a+O(\sigma^2)}, \quad B=-\frac{\sigma b+\sigma^2(ab+c^2)+O(\sigma^3)}{1-\sigma^2 a+O(\sigma^3)}.
\]
Since $B<0$ for small $\sigma>0$, we have $\psi(x)\to-\infty$ as $x\to\infty$. This and the condition $\psi_0(0)=1$ complete the proof.

\end{proof}

{\it Proof of Theorem \ref{tl}.} Let $\phi=\psi_0$ for $x\leq x_0,~ \phi=0$ for $x>x_0$. Then, for arbitrarily small $\sigma>0$,
\[
0=<H\phi,\phi>=\|\phi'\|_{L^2(\mathbb R_+)}-\sigma\int_0^\infty V(x)\phi^2(x)dx.
\]
Since the first term on the right is positive, the second one is negative, and therefore the quadratic form is negative when $\sigma$ is replaced by $2\sigma$. Hence $H$ with $\sigma$ replaced by $2\sigma$ has a negative eigenvalue. Since the arguments used in the first paragraph of the proof of Theorem \ref{t2} imply that it is enough to justify the statement of Theorem \ref{tl} only for small $\sigma>0$, the proof is complete.

\qed

\section{Spectrum of a Schr\"{o}dinger operator with $V\in L^1(\mathbb{R}^2)$}
This section is devoted to a construction of an example showing that the essensial spectrum of the Schr\"{o}dinger operator
\begin{equation}\label{ws}
H=-\Delta-V(x), \quad x\in \mathbb{R}^2, \quad V\in L^1(\mathbb{R}^2),
\end{equation} 
may cover the whole spectral axis. 

We need the folowing lemma.
Denote by $v(h,\delta, |x|)$ the potential
\begin{equation}\label{pot}
V=v(h,\delta, |x|)=\left\{\begin{array}{c}
                    h, ~~|x|<\delta, \\
                    0, ~~|x|\geq \delta.
                  \end{array}\right.
\end{equation}
\begin{lemma}\label{ll} There is a constant $C$ such that for
each $\lambda<0$ and $\delta\leq1/\sqrt{|\lambda|}$ there exists $h=h(\delta,\lambda)$ for which
\begin{equation}\label{h}
h\delta^2\leq |\lambda|\delta^2+\frac{C}{\ln(\sqrt{|\lambda|}\delta)^{-1}}
\end{equation}
and operator (\ref{ws}) with the potential (\ref{pot}) has an exponentially decaying at infinity eigenfunction $\psi$ with the eigenvalue $\lambda$.
\end{lemma}
\begin{proof} We look for $\psi$ in the form
\begin{equation}\label{psi1}
\psi( |x|)=\left\{\begin{array}{c}
                    \frac{J_0(\sqrt{h+\lambda}|x|)}{J_0(\sqrt{h+\lambda}\delta)}, ~~|x|<\delta, \\
                    \frac{K_0(\sqrt{|\lambda|}|x|)}{K_0(\sqrt{|\lambda|}\delta)}, ~~|x|\geq \delta,
                  \end{array}\right.
\end{equation}
where $h>|\lambda|, ~J_0$ is the Bessel function and $K_0$ is the modified Bessel function (i.e., $K_0$ is proportional to the Hankel function of the complex argument). Then $-\Delta\psi -v\psi=\lambda\psi, ~|x|\neq\delta$. Since $\psi$ is continuous, the continuity of $\nabla\psi$ implies that $\psi$ is an eigenfunction of $H=-\Delta-v(x)$. Hence, $\psi$ is an eigenfunction if  
\begin{equation}\label{be}
\frac{ \tau J'_0(\tau)}{J_0(\tau)}=\frac{\sqrt{|\lambda|}\delta K'_0(\sqrt{|\lambda|}\delta)}{ K_0(\sqrt{|\lambda|}\delta)}, \quad \tau=\sqrt{h+\lambda}\delta.
\end{equation}
Function $K_0$ decays exponentially at infinity, and therefore the same is true for function (\ref{psi1}). Thus, it remains to show that (\ref{be}) has a solution for which (\ref{h}) holds. 

Since $K_0(x)>0$ and $K'_0(x)<0$ for all $x>0$, the right-hand side above is negative when $\sqrt{|\lambda|}\delta>0$. The function on the left changes monotonocally from $0$ to $-\infty$ when $0<\tau<a$, where $a$ is the first positive root of $J_0(\tau)$.
  Hence, equation (\ref{be}) has a continuous solution $\tau=\tau_0(\sqrt{|\lambda|}\delta)\in(0,a),~\sqrt{|\lambda|}\delta>0$. Its behavior when $\sqrt{|\lambda|}\delta\to0$ can be found from the asymptotics of the Bessel functions. The right-hand side  in (\ref{be}) equals $-O(|\ln^{-1}(\sqrt{|\lambda|}\delta)|), ~\sqrt{|\lambda|}\delta\to0$, and the left-hand side is $\frac{-\tau^2}{2}(1+O(\tau^2)),~\tau\to0$.  Hence, $\tau_0^2=O(|\ln^{-1}(\sqrt{|\lambda|}\delta)|), ~\sqrt{|\lambda|}\delta\to0$, and this implies (\ref{h}). 

\end{proof}

\begin{theorem}
For each $\varepsilon>0$, one can find a potential $V=V_\varepsilon\in L^1(\mathbb{R}^2)$ such that $\|V_\varepsilon\|_{L^1}\leq \varepsilon$ and the essential spectrum of operator (\ref{ws}) covers the whole spectral axis.
\end{theorem}
\begin{proof} We fix an arbitrary sequence $\{\lambda_n\}$ of positive numbers that is dense on $(0,\infty)$. We fix a sequence $\{\delta_n>0\},~n=1,2...$, for which $h_n$ defined in Lemma \ref{ll} with $\lambda=\lambda_n$  has an estimate $h_n\delta_n^2<n^{-2}$.
Let $V_n=v(h_n,\delta_n,|x|)$ be the potential (\ref{pot}), and let $\psi_n$ be the eigenfunction with the eigenvalue $\lambda=\lambda_n$ constructed in Lemma \ref{ll} when $h=h_n,\delta=\delta_n$.
 The potential $-V$ will be a sum of sparse, narrow and deep potential wells:
\begin{equation}\label{ve}
-V(x)=-\sum_{n=n_0}^\infty v(h_n,\delta_n,|x-x_n|).
\end{equation}

We choose $n_0$ for which $\sum_{n\geq n_0}n^{-2}<\varepsilon$. Then $\|V\|_{L^1}<\varepsilon$. We fix a function $\chi=\chi(\tau)\in C^\infty(\mathbb R)$ such that $\chi(\tau)=1$ for $\tau\leq-1,~\chi(\tau)=0$ for $\tau\geq0,$ and define
\[
\phi_n(|x|)=\frac{\psi_n(|x|)}{\|\psi_n\|_{L^2(\mathbb R^2)}}\chi(|x|-R_n)
\]
with $R_n$ increasing so fast that
\begin{equation}\label{pr}
\|\phi_n\|_{L^2(\mathbb R^2)}\to 1 \quad {\rm as}~~n\to\infty, \quad \|(-\Delta-v(x)-\lambda_n)\phi_n\|_{L^2(\mathbb R^2)}\to 0 \quad {\rm as}~~n\to\infty.
\end{equation}
Such a sequence $\{R_n\}$ exists since functions $\psi_n$ decay exponentially at infinity.
Then we choose points $x_n$ in (\ref{ve}) in such a way that circles $B_n$ of radius $R_n$ centered at $x=x_n$ do not intersect each other.

We fix an arbitrary $\lambda_0\leq0$ and an arbitrary subsequence $\phi_{n_j}$ of functions $\phi_{n}$ for which $\lambda_{n_j}\to\lambda_0$
as $j\to\infty$. Functions $\phi_{n_j}(|x-x_{n_j}|), j=1,2...,$ are orthogonal since their supports belong to different balls $B_{n_j} $. Relations (\ref{pr}) imply that $\|\phi_{n_j}\|_{L^2(\mathbb R^2)}\to 1$ as $j\to\infty$ and
\[
\|(-\Delta-V(x)-\lambda_0)\phi_{n_j}(|x-x_{n_j}|)\|_{L^2}=\|(-\Delta-v(h_{n_j},\delta_{n_j},|x-x_{n_j}|)-\lambda_{n_j})\phi_{n_j}(|x-x_{n_j}|)\|_{L^2}
\]
\[
+(\lambda_{n_j}-\lambda_0)\|\phi_{n_j}(|x-x_{n_j}|)\|_{L^2}\to 0 \quad {\rm as} ~~j\to\infty.
\]
Thus the Weyl criterion implies that $\lambda_0$ belongs to the essential spectrum. The point $\lambda_0\in(-\infty,0]$ is arbitrary, and therefore the essential spectrum contains the semi-axis $(-\infty,0]$.

A Weyl sequence for $\lambda>0$ can be constructed in a standard way using truncated plane waves whose support does not intersect the support of $V$.
 
\end{proof}

\section{Examples concerning transient operators}
Let
\begin{equation}\label{h3a}
H\psi=-\Delta\psi-\sigma V(x)\psi, \quad x\in \mathbb{R}^d, ~~d\geq 3.
\end{equation}
Condition (\ref{clr}) is a standard requirement for CLR-type theorems that provide an estimate from above on the number $N(\sigma V)$ of negative eigenvalues of $H$ and for
the quasi-classical asymptotics of $N(\sigma V)$.
 Below we provide a couple of examples concerning the situation when (\ref{clr}) does not hold.

{\bf Example 1 (Large separated spots).} We consider operator (\ref{h3a}), where $V(x)$ is an arbitrary potential such that $V(x)\geq\frac{1}{|x|^2}$ on a sequence of disjoint balls with fast-increasing radii. Then $N(\sigma V)$ can be infinite for large enough $\sigma$.

Let us provide a particular example of this situation. We fix arbitrary points $x_n\in \mathbb{R}^d, ~d\geq3,$ such that $|x_n|=2^n$, choose $R_n=
\frac{1}{3}2^n$, and consider the balls $B_n=\{x: |x-x_n|<R_n\}$.
In order to show that the number of negative eigenvalues is infinite, consider the test functions
$\psi_n(x)=\psi(\frac{|x-x_n|}{R_n})$, where $\psi=\psi(x)$ is a fixed infinitely smooth function with the support in the unit ball $|x|<1$.
Then
\[
<\Delta\psi_n,~\psi_n>=\int_{\mathbb{R}^d}\Delta\psi_n(x)\psi_n(x)dx=R_n^{d-2}\int_{\mathbb{R}^d}\Delta\psi(x)\psi(x)dx=C_1R_n^{d-2},
\]
\[
<V\psi_n,~\psi_n>=\int_{\mathbb{R}^d}V(x)|\psi_n(x)|^2dx\geq R_n^{d-2}\int_{\mathbb{R}^d}|x|^{-2}|\psi(x)|^2dx=C_2R_n^{d-2},
\]
and therefore $<H\psi_n,~\psi_n>$ is negative for all $n\geq 1$ if $\sigma>|C_1|/C_2$. Then the variational principle implies that $N(\sigma V)=\infty$.

{\bf Example 2 (Small sparse spots).} Here we will provide an example of a non-negative potential $V(x)$ for which operator (\ref{h3a}) with $\sigma\in(0,1)$ does not have negative eigenvalues in spite of the fact that $V$ decays at infinity so slowly that
\begin{equation}\label{beta}
\int_{\mathbb{R}^d}V^\beta(x)dx=\infty   \quad {\rm for}~{\rm each} ~~~\beta>0.
\end{equation}
The potential $V$ will be a sum of sparse bumps.

We choose points $x_n$ as in the previous example (i.e., $|x_n|=2^n$). Let $V_0\in C_0^\infty$ (an elementary bump) have the support in the unit ball $|x|<1$. We choose
\begin{equation}\label{sp}
V(x)=\sum _{n_0}^\infty \frac{1}{\ln n}V_0(x-x_n),
\end{equation}
where $n_0\geq3$ will be chosen later. Since $V(x)\to0, ~|x|\to\infty,$ the negative spectrum of $H$ is discrete. Hence, if the negative spectrum is not empty, there is a ground state $\psi_0(x)>0$ that corresponds to
the smallest negative eigenvalue $\lambda_0<0$. Obviously, (\ref{beta}) holds for potential (\ref{sp}), and it remains to show that the negative spectrum of $H$ is empty for $\sigma\in(0,1)$.

If $\psi_0(x)$ exists, then it satisfies the equation
\begin{equation}\label{rel}
\psi_0(x)=\int_{\mathbb{R}^d} G_{\lambda_0}(x-y)\sigma V(y-x_n)\psi_0(y)dy=\sum _{n_0}^\infty \int_{|y-x_n|<1}G_{\lambda_0}(x-y)\frac{\sigma}{\ln n}V_0(y-x_n)\psi_0(y)dy,
\end{equation}
where $G_{\lambda_0}(x-y)$ is the integral kernel of the resolvent $(\Delta-\lambda_0)^{-1}$. Denote
\[
M_n=\max_{x: |x-x_n|\leq 1}\psi_0(x).
\]
From local a priori estimates for solutions of the Laplace equation it follows that $|\psi_0|$ is bounded, and therefore $\{M_n\}\in l^\infty$. Since $0<G_{\lambda_0}(x)<c(d)/|x|^{d-2}, ~x\in \mathbb R^d,$ equation (\ref{rel}) with $\sigma\in (0,1)$ implies that
\[
M_m\leq\frac{1}{\ln n_0}\sum\limits_{n\geq n_0, ~n\neq m}\frac{v(d)c(d)M_n}{|2^m-2^n|^{d-2}-2}+\frac{c_1(d)}{\ln m}M_m
\]
\[
:=\frac{1}{\ln n_0}\sum\limits_{n\geq n_0 ~n\neq m}\Gamma_{n,m}M_n+\gamma_mM_m, \quad  m\geq n_0,
\]
where $v(d)$ is the volume of the unit ball in $\mathbb R^d$. We choose $n_0$ so large that $\gamma_m<1/2$ for $m\geq n_0$. Then
\[
M_m\leq\frac{2}{\ln n_0}\sum\limits_{n\geq n_0 ~n\neq m}\Gamma_{n,m}M_n, \quad  m\geq n_0,
\]
where $\rho:=\max_m\sum _{n\geq 3}\Gamma_{n,m}<\infty$. Hence the norm of the operator with the matrix elements $\Gamma_{n,m},~n,m\geq 3,~n\neq m,$ in the space $l^\infty$ of bounded sequences is bounded, and therefore the system of inequalities for $M_m$ has only trivial solution if $n_0$ is large enough.
Thus the negative spectrum is empty.

\noindent {\bf \large Acknowledgments}:
The work of B. Vainberg was supported by the Simons Foundation grant 527180.

\end{document}